\documentclass[12pt,a4paper]{article}%
\usepackage[utf8]{inputenc}
\usepackage{amsmath}
\usepackage{amsfonts}
\usepackage{amssymb}
\usepackage{graphicx}
\usepackage[space]{grffile}
\usepackage{hyperref}%
\providecommand{\U}[1]{\protect\rule{.1in}{.1in}}
\newtheorem{theorem}{Theorem}
\newtheorem{theorem*}{Example}

\newtheorem{conjecture}[theorem]{Conjecture}

\newtheorem{lemma}[theorem]{Lemma}

\newtheorem{remark}[theorem]{Observation}

\newenvironment{proof}[1][Proof]{\noindent\textbf{#1.} }{\ \hfill \rule{0.5em}{0.5em}\bigskip}
\graphicspath{{D:/Dropbox/Riste-Sedlar/Totally Irregular/Paper3/Slike/}}
\begin{document}

\title{A note on the locally irregular edge colorings of cacti}
\author{Jelena Sedlar$^{1,3}$,\\Riste \v Skrekovski$^{2,3}$ \\[0.3cm] {\small $^{1}$ \textit{University of Split, Faculty of civil
engineering, architecture and geodesy, Croatia}}\\[0.1cm] {\small $^{2}$ \textit{University of Ljubljana, FMF, 1000 Ljubljana,
Slovenia }}\\[0.1cm] {\small $^{3}$ \textit{Faculty of Information Studies, 8000 Novo
Mesto, Slovenia }}\\[0.1cm] }
\maketitle

\begin{abstract}
A graph is locally irregular if the degrees of the end-vertices of every edge
are distinct. An edge coloring of a graph $G$ is locally irregular if every
color induces a locally irregular subgraph of $G$. A colorable graph $G$ is
any graph which admits a locally irregular edge coloring. The locally
irregular chromatic index $\chi_{%
\rm{irr}%
}^{\prime}(G)$ of a colorable graph $G$ is the smallest number of colors
required by a locally irregular edge coloring of $G$. The Local Irregularity
Conjecture claims that all colorable graphs require at most $3$ colors for a
locally irregular edge coloring. Recently, it has been observed that the
conjecture does not hold for the bow-tie graph $B$, since $B$ is colorable and
requires at least $4$ colors for a locally irregular edge coloring. Since $B$
is a cactus graph and all non-colorable graphs are also cacti, this seems to
be a relevant class of graphs for the Local Irregularity Conjecture. In this
paper we establish that $\chi_{%
\rm{irr}%
}^{\prime}(G)\leq4$ for all colorable cactus graphs.

\end{abstract}



\section{Introduction}

All graphs mentioned in this paper are considered to be simple and finite. A
\emph{cactus graph} is any graph with edge disjoint cycles. A graph is said to
be \emph{locally irregular} if the degrees of the two end-vertices of every
edge are distinct. A \emph{locally irregular }$k$\emph{-edge coloring}, or
$k$\emph{-liec} for short, is any $k$-edge coloring of $G$ every color of
which induces a locally irregular subgraph of $G$. Since in this paper we deal
only with the locally irregular edge colorings, a graph which admits such a
coloring will be called \emph{colorable}. The \emph{locally irregular
chromatic index} $\chi_{%
\rm{irr}%
}^{\prime}(G)$ of a colorable graph $G$ is defined as the smallest $k$ such
that $G$ admits a $k$-liec. The first question is which graphs are colorable?
To answer this question, we first need to introduce a special class
$\mathfrak{T}$ of cactus graphs. The class $\mathfrak{T}$ is defined as follows:

\begin{itemize}
\item[R1.] The triangle $K_{3}$ is contained in $\mathfrak{T}$;

\item[R2.] For every graph $G\in\mathfrak{T}$, a graph $H$ which also belongs
to $\mathfrak{T}$ can be constructed in the following way: a vertex $u\in
V(G)$ of degree $2,$ which belongs to a triangle of $G,$ is identified with an
end-vertex of an even length path or with the only vertex of degree one in a
graph consisting of a triangle and an odd length path hanging at one of the
vertices of triangle.
\end{itemize}

\noindent Obviously, graphs from $\mathfrak{T}$ are subcubic cacti in which
cycles are vertex disjoint triangles, which are connected by an odd length
paths, and besides triangles and odd length paths connecting them a cactus
graph $G$ from $\mathfrak{T}$ may have only even length paths hanging at any
vertex $u$ such that $u$ belongs to a triangle of $G$ and $d_{G}(u)=3.$ It was
established that the only non-colorable graphs are odd length paths, odd
length cycles and cacti from $\mathfrak{T}$ \cite{Baudon5}. Before we proceed,
let us make the following straightforeward observation which will be used later.

\begin{remark}
\label{Observation_G0}Let $G$ be a non-colorable graph and let $e\in E(G).$ If
$e$ is an edge incident to a leaf or $e$ belongs to a cycle of $G,$ then $G-e$
is colorable.
\end{remark}

For colorable graphs an interesting question is what is the smallest number of
colors required by a locally irregular edge coloring of any graph. Regarding
this question, the following conjecture was proposed \cite{Baudon5}.

\begin{conjecture}
[Local Irregularity Conjecture]\label{Con_nonTareColorable}For every colorable
connected graph $G$, it holds that $\chi_{%
\rm{irr}%
}^{\prime}(G)\leq3.$
\end{conjecture}

It was recently shown \cite{SedSkreMasna} that the Local Irregularity
Conjecture does not hold in general, since there exists a colorable cactus
graph, the so called bow-tie graph $B$ shown in Figure \ref{Fig_masna}, for
which $\chi_{%
\rm{irr}%
}^{\prime}(G)=4$. There, the following weaker version of Conjecture
\ref{Con_nonTareColorable} was proposed.

\begin{figure}[h]
\begin{center}
\includegraphics[scale=0.85]{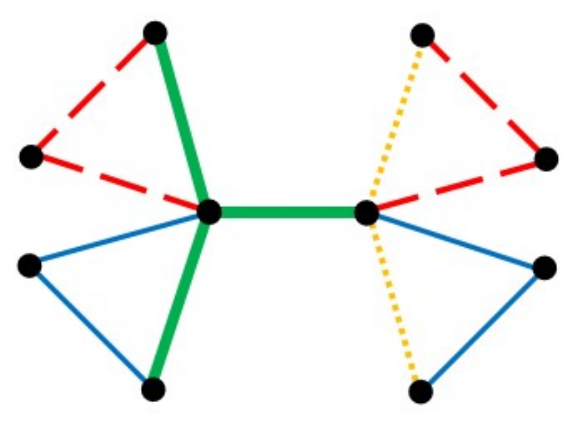}
\end{center}
\caption{The bow-tie graph $B$ and a $4$-liec of it.}%
\label{Fig_masna}%
\end{figure}

\begin{conjecture}
\label{Con1}Every colorable connected graph $G$ satisfies $\chi_{%
\rm{irr}%
}^{\prime}(G)\leq4.$
\end{conjecture}

Even though the graph $B$ contradicts Local Irregularity Conjecture, many
partial results support this conjecture. For example, the conjecture holds for
trees \cite{Baudon6}, unicyclic graphs and cacti with vertex disjoint cycles
\cite{SedSkreMasna}, graphs with minimum degree at least $10^{10}$
\cite{Przibilo}, $k$-regular graphs where $k\geq10^{7}$ \cite{Baudon5}. The
conjecture was also considered for general graphs, but there the upper bound
on the number of colors required by a liec of a colorable $G$ was first
established to be $328$ \cite{Bensmail}, and then it was lowered to $220$
\cite{Luzar}.

The results so far indicate that cacti are the relevant class of graphs for
the locally irregular edge colorings, since the non-colorable graphs are cacti
and the only known counterexample, graph $B$, for the Local Irregularity
Conjecture is also a cactus graph. Motivated by this, in this paper we will
further consider this class of graphs. We will show that all colorable cacti
satisfy Conjecture \ref{Con1}.

\section{Preliminaries}

Let us introduce some results already known in literature, mainly regarding
trees and unicyclic graphs, which will be of use in the rest of the paper.

A tree rooted at its leaf will be called a \emph{shrub}. The \emph{root edge}
of a shrub $T$ is the edge incident to the root vertex of $T.$ An edge
coloring of a shrub $T$ is said to be an \emph{almost locally irregular }%
$k$\emph{-edge coloring}, or $k$\emph{-aliec} for short, if either it is a
$k$-liec of $T$ or it is a coloring such that only the root edge is locally
regular. Let us state few results regarding trees from \cite{Baudon6}.

\begin{theorem}
\label{Tm_BaudonSchrub}Every shrub admits a $2$-aliec.
\end{theorem}

\begin{theorem}
\label{Tm_BaudonTree}For every colorable tree $T,$ it holds that $\chi_{%
\rm{irr}%
}^{\prime}(T)\leq3.$ Moreover, $\chi_{%
\rm{irr}%
}^{\prime}(T)\leq2$ if $\Delta(T)\geq5.$
\end{theorem}

Beside Theorems \ref{Tm_BaudonSchrub} and \ref{Tm_BaudonTree}, we will also
need some specific claims used in \cite{Baudon6} to prove our results. In
order to state them, let us introduce some more notation. The colors of an
edge coloring will be denoted by letters $a,$ $b,$ $c,\ldots$ By $\phi_{a}(G)$
(resp. $\phi_{a,b}(G),$ $\phi_{a,b,c}(G)$) we will usually denote an edge
coloring of $G$ using one (resp. two, three) colors. For an edge coloring
$\phi_{a,b,c}$ of $G$ and a vertex $u\in V(G),$ by $\phi_{a,b,c}(u)$ we will
denote the set of colors incident to $u.$ If $\phi_{a,b,c}(u)$ contains $k$
colors, then we say $u$ is $k$\emph{-chromatic}, specifically in the cases of
$k=1$ and $k=2$ we say $u$ is \emph{monochromatic} and \emph{bichromatic},
respectively. If $\phi_{a,b}$ is an edge coloring of a graph $G$ which uses
colors $a$ and $b,$ then $\phi_{c,d}$ will denote the edge coloring of $G$
obtained from $\phi_{a,b}$ by replacing color $a$ with $c$ and color $b$ with
$d$. In particular, the edge coloring $\phi_{b,a}$ obtained from $\phi_{a,b}$
by swapping colors $a$ and $b$ will be called the \emph{inversion} of
$\phi_{a,b}.$

We also need to introduce a \emph{sum of colorings}, useful when combining two
different colorings. Let $G_{i},$ for $i=1,\ldots,k,$ be graphs with pairwise
disjoint edge sets, and let $G$ be a graph such that $E(G)=\cup_{i=1}%
^{k}E(G_{i}).$ If $\phi_{a,b,c}^{i}$ is an edge coloring of the graph $G_{i}$
for $i=1,\ldots,k,$ then $\phi_{a,b,c}=\sum_{i=1}^{k}\phi_{a,b,c}^{i}$ will
denote the edge coloring of $G$ such that $e\in E(G_{i})$ implies
$\phi_{a,b,c}(e)=\phi_{a,b,c}^{i}(e).$

The $a$\emph{-degree} of a vertex $u\in V(G)$ is defined as the number of
edges incident to $u$ which are colored by $a$, and it is denoted by
$d_{G}^{a}(u).$ The same name and notation goes for any other color besides
$a$. Now, for a vertex $u\in V(G)$ with $k$ incident edges colored by $a$, say
$uv_{1},\ldots,uv_{k},$ the $a$\emph{-sequence }is defined as $d_{G}^{a}%
(v_{1}),\ldots,d_{G}^{a}(v_{k}).$ It is usually assumed that vertices $v_{i}$
are denoted so that the $a$-sequence is non-increasing.

Regarding shrubs, they admit an $2$-aliec according to Theorem
\ref{Tm_BaudonSchrub}, which will be denoted by $\phi_{a,b}$ where we assume
the root edge of the shrub is colored by $a$. Now, let $T$ be a tree, $u\in
V(T)$ a vertex of maximum degree in $T,$ and $v_{i}$ all the neighbors of $u$
for $i=1,\ldots,k.$ Denote by $T_{i}$ a shrub of $T$ rooted at $u$ with the
root edge $uv_{i}.$ A \emph{shrub based} coloring of $T$ is defined by
$\phi_{a,b}=\sum_{i=1}^{k}\phi_{a,b}^{i},$ where $\phi_{a,b}^{i}$ is an
$2$-aliec of $T_{i}.$ Since we assume that the root edge of $T_{i}$ is colored
by $a$ in $\phi_{a,b}^{i},$ this implies $u$ is monochromatic in color $a$ by
a shrub based coloring $\phi_{a,b}$ of $T.$ Obviously, if a shrub based
coloring $\phi_{a,b}$ is not a liec of $T,$ only the edges incident to $u$ may
be locally regular by $\phi_{a,b}.$ Notice that $\binom{k}{2}$ different
$2$-edge colorings of $T$ can be obtained from a shrub based coloring
$\phi_{a,b}$ by swapping colors $a$ and $b$ in some of the shrubs $T_{i}$. If
none of those colorings is a liec of $T,$ we say $\phi_{a,b}$ is
\emph{inversion resistant}. The following observation was established in
\cite{Baudon6}.

\begin{remark}
\label{Observation_sequences}Let $T$ be a tree, $u\in V(T)$ a vertex of
maximum degree in $T$ and $\phi_{a,b}$ a shrub based coloring of $T$ rooted at
$u.$ The shrub based coloring $\phi_{a,b}$ will be inversion resistant in two
cases only:

\begin{itemize}
\item $d_{T}(u)=3$ and the $a$-sequence of $u$ by $\phi_{a,b}$ is $3,2,2$;

\item $d_{T}(u)=4$ and the $a$-sequence of $u$ by $\phi_{a,b}$ is $4,3,3,2$.
\end{itemize}
\end{remark}

The obvious consequence of the above observation is that in a tree $T$ with
$\chi_{%
\rm{irr}%
}^{\prime}(T)=3$ and $\Delta(T)=3$ (resp. $\Delta(T)=4$) the vertices of
degree $3$ (resp. $4$) must come in neighboring pairs. Also, if a shrub based
coloring of a tree $T$ rooted at a vertex $u$ of maximum degree is inversion
resistant, then $\chi_{%
\rm{irr}%
}^{\prime}(T)=3.$ A $3$-liec of such a tree $T$ is obtained from aliecs of
shrubs in a following way: if $d_{T}(u)=3$ then $\phi_{a,b,c}^{T}=\phi
_{a,b}^{1}+\phi_{b,a}^{2}+\phi_{c,b}^{3}$ is a $3$-liec of $T,$ if
$d_{T}(u)=4$ then $\phi_{a,b,c}^{T}=\phi_{a,b}^{1}+\phi_{a,b}^{2}+\phi
_{b,a}^{3}+\phi_{c,b}^{4}$ is a $3$-liec of $T.$ Notice that in this $3$-liec
of $T,$ only the root vertex $u$ is $3$-chromatic, all other vertices in $T$
are $1$- or $2$-chromatic. Hence, we will call $u$ the \emph{rainbow root} of
such a liec. Obviously, every vertex $u$ of maximum degree of a tree $T$ with
$\chi_{%
\rm{irr}%
}^{\prime}(T)=3$ can be the rainbow root of a $3$-liec of $T$, because either
the degree sequence of $u$ by a shrub based coloring is $3,2,2$ (resp.
$4,3,3,2$) which means the above $3$-liecs can be constructed with $u$ being
the rainbow root, or the shrub based coloring would not be inversion
resistant, which implies $\chi_{%
\rm{irr}%
}^{\prime}(G)\leq2,$ a contradiction.\ 

Further, notice that the color $c$ is used by $\phi_{a,b,c}^{T}$ in precisely
one shrub of $T.$ Simply, one can choose that one shrub to be any of the
shrubs of $T$ rooted at $u$. Let us now collect all this in the following
formal observation for further referencing.

\begin{remark}
\label{Observation_maxDeg}Let $T$ be a colorable tree with $\chi_{%
\rm{irr}%
}^{\prime}(T)=3.$ Then all vertices of maximum degree in $T$ come in
neighboring pairs. Also, every vertex of maximum degree in $T$ can be the
rainbow root of a $3$-liec of $T$. Finally, for any vertex $u$ of maximum
degree in $T$ and for any shrub $T_{i}$ of $T$ rooted at $u,$ there exists a
$3$-liec of $T$ such that the color $c$ is used only in $T_{i}$.
\end{remark}

Finally, we will also need the following result from \cite{SedSkreMasna}\ on
unicyclic graphs.

\begin{theorem}
\label{Tm_unicyclic}Let $G$ be a colorable unicyclic graph. Then $\chi_{%
\rm{irr}%
}^{\prime}(G)\leq3.$
\end{theorem}

\section{Coloring cacti by four colors}

A \emph{grape} $G$ is any cactus graph with at least one cycle in which all
cycles share a vertex $u,$ and the vertex $u$ is called the \emph{root} of
$G.$ A \emph{berry} $G_{i}$ of a grape $G$ is any subgraph of $G$ induced by
$V(G_{i}^{\prime})\cup\{u\},$ where $u$ is the root of $G$ and $G_{i}^{\prime
}$ a connected component of $G-u.$ Notice that a berry $G_{i}$ can be either a
unicyclic graph in which $u$ is of degree $2$ or a tree in which $u$ is a
leaf, so such berries will be called \emph{unicyclic berries} and
\emph{acyclic berries}, respectively. A unicyclic berry $G_{i}$ is said to be
\emph{triangular} if its cycle is the triangle.

\begin{figure}[h]
\begin{center}
\includegraphics[scale=0.75]{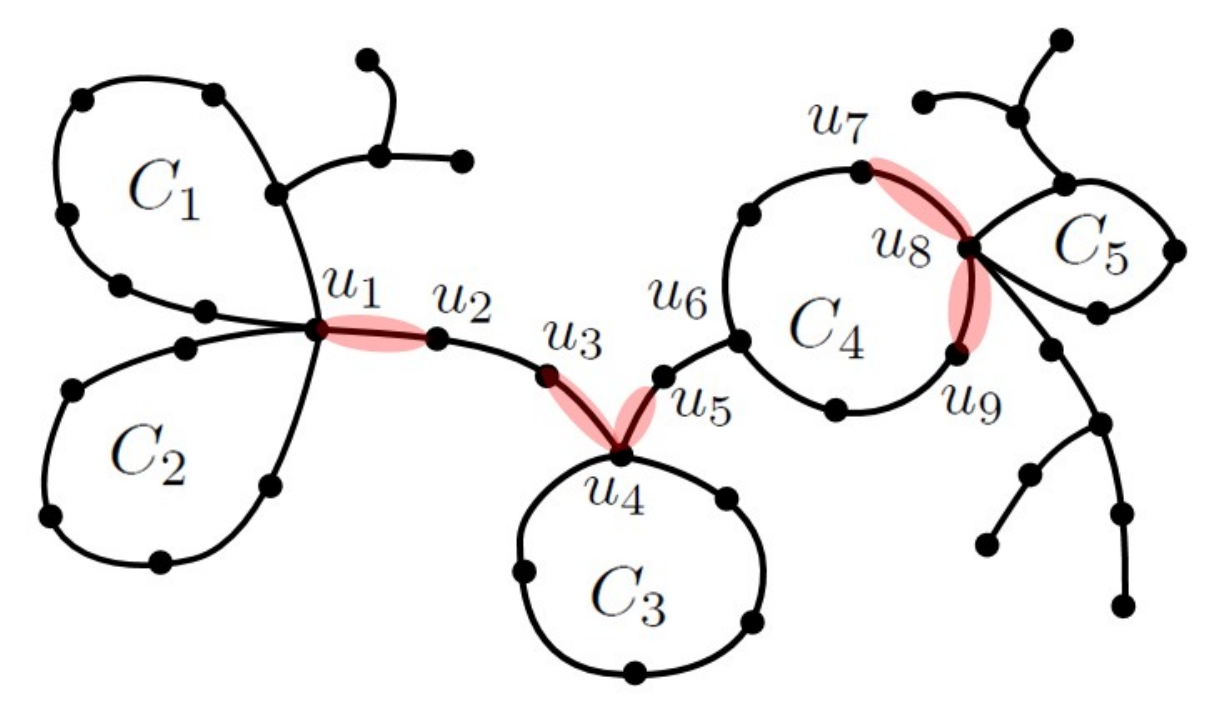}
\end{center}
\caption{A cactus graph $G$ with five cycles which contains two end-grapes,
$G_{u_{1}}$ and $G_{u_{8}}.$ The end-grape $G_{u_{1}}$ has two unicyclic
berries and one exit edge $u_{1}u_{2}.$ The end-grape $G_{u_{8}}$ consists of
one unicyclic and one acyclic berry and it has two exit edges $u_{8}u_{7}$ and
$u_{8}u_{9}.$ Notice that the cycle $C_{3}$ is not an end-grape of $G$ since
it has two exit edges $u_{4}u_{3}$ and $u_{4}u_{5}$ and they do not belong to
a same cycle.}%
\label{Fig_endCycles}%
\end{figure}

An \emph{end-grape} $G_{u}$ of a cactus graph $G$ is any subgraph of $G$ such that:

\begin{itemize}
\item $G_{u}$ is a grape rooted at $u$ where $u$ is the only vertex of $G_{u}$
incident to edges from $G-E(G_{u}),$ and

\item $u$ is incident to either one edge from $G-E(G_{u})$ or two such edges
which then must belong to a same cycle of $G-E(G_{u}),$ and such edges are
called the \emph{exit edges} of $G_{u}$.
\end{itemize}

\noindent This notion is illustrated by Figure \ref{Fig_endCycles}. Also, for
an end-grape $G_{u}$ rooted at $u,$ the graph $G_{0}=G-E(G_{u})$ will be
called the \emph{root component} of $G_{u}$. Notice that $d_{G_{0}}(u)\leq2$
and also that $E(G)=E(G_{0})\cup E(G_{u})$ where $E(G_{0})\cap E(G_{u}%
)=\emptyset.$ Let us first prove the following auxiliary result.

\begin{lemma}
\label{Lemma_nonColorable}Let $G$ be a colorable cactus graph with $c\geq2$
cycles and $G_{u}$ an end-grape of $G$. If $G_{u}$ is comprised of just one
berry and that berry is a non-colorable triangular berry which has only one
exit edge, then there exists a colorable cactus graph $G_{0}^{\prime}$ with
fewer cycles than $G,$ such that $\chi_{%
\rm{irr}%
}^{\prime}(G)\leq\max\{\chi_{%
\rm{irr}%
}^{\prime}(G_{0}^{\prime}),3\}.$
\end{lemma}

\begin{proof}
Denote by $v$ and $w$ be the neighbors of $u$ in $G_{u}.$ Since $G_{u}$ is
non-colorable triangular berry, it consists of the triangle $uvw$ and possibly
an even length path hanging at $v$ and/or $w$. Let $G_{0}^{\prime}$ be the
connected component of $G-v$ which contains $u.$ Notice that $G_{0}^{\prime}$
consists of a root component $G_{0}$ of $G_{u}$ and a pendant odd length path
hanging at $u$ which contains $w$. If $G_{0}^{\prime}$ is non-colorable, since
it contains a pendant path we conclude $G_{0}^{\prime}\in\mathfrak{T},$ but
then $G$ is also non-colorable, a contradiction. So, we may assume
$G_{0}^{\prime}$ is colorable and admits a liec which uses say colors $a$ and
$b,$ so we will denote it by $\phi_{a,b}^{0}.$ Since $u$ and $w$ belong to a
pendant path $P$ of $G_{0}^{\prime},$ given the parity of their distance to
the leaf of the path $P$, we may assume $\phi_{a,b}^{0}(u)=\{a\}$ and
$\phi_{a,b}^{0}(w)\subseteq\{a,b\}.$ Let $T=G-E(G_{0}^{\prime})$ and notice
that $T$ is a tree which consists of the path $uvw$ with possibly one even
length path hanging at $v.$ Obviously, $T$ admits a $2$-liec $\phi_{b,c}^{T}$
such that $\phi_{b,c}^{T}(u)=\phi_{b,c}^{T}(w)=\{c\}.$ Since $G$ consists of
$G_{0}^{\prime}$ and $T$ which meet at vertices $u$ and $w,$ we conclude that
$\phi_{a,b}^{0}+\phi_{b,c}^{T}$ is a $k$-liec of $G$, where $k=\max\{\chi_{%
\rm{irr}%
}^{\prime}(G_{0}^{\prime}),3\}.$
\end{proof}

Before we proceed with the main result, let us introduce the notation we will
use in a cactus graph with an end-grape $G_{u}$ and its root component
$G_{0}.$ Let $p$ (resp. $q$) denote the number of cyclic (resp. acyclic)
berries in $G_{u}$. We may assume berries of $G_{u}$ are denoted by $G_{i}$,
for $i=1,\ldots,p+q,$ so that $G_{i}$ is a cyclic berry if and only if $i\leq
p.$ Let $v_{i}$ be a neighbor of $u$ in a berry $G_{i}$ for $i=1,\ldots,p+q.$
Let $w_{i}$ be the other neighbor of $u$ in a cyclic berry $G_{i}$ for
$i=1,\ldots,p.$ The neighbors of $u$ in the root component $G_{0}$ will be
denoted by $u_{1}$ and $u_{2},$ where $u_{2}$ does not exist if $d_{G_{0}%
}(u)=1.$

\begin{theorem}
\label{Theorem_main}Let $G$ be a colorable cactus graph. Then $\chi_{%
\rm{irr}%
}^{\prime}(G)\leq4$\emph{.}
\end{theorem}

\begin{proof}
The proof is by the induction on the number of cycles in $G.$ If $G$ contains
less than $2$ cycles, then $G$ is a tree or a unicyclic graph, so the claim
follows from Theorems \ref{Tm_BaudonTree} and \ref{Tm_unicyclic},
respectively. Let us assume that the claim holds for all cacti with less than
$c\geq2$ cycles, and let $G$ be a cactus graph with $c$ cycles.

If $G$ is a grape with the root vertex $u,$ let $E_{u}$ be the set consisting
of an edge incident to $u$ from every cycle in $G.$ Since $G$ contains at
least two cycles, $E_{u}$ contains at least two edges, hence $E_{u}$ induces a
locally irregular subgraph of $G.$ Notice that $T=G-E_{u}$ is a tree. If $T$
is colorable, then it can be colored by $3$ colors, so coloring $E_{u}$ by the
fourth color yields a $4$-liec of $G$. If $T$ is non-colorable, then $T$ is a
path, and there exists an edge $uz$ in $T$ such that $uz\not \in E_{u}$ and
$T^{\prime}=T-uz$ is a collection of one or two even length paths. Thus
$T^{\prime}$ is colorable and can be colored by $2$ colors, then coloring
$E_{u}^{\prime}=E_{u}\cup\{uz\}$ by the third color yields a $3$-liec of $G$.

Assume now that $G$ is not a grape, so let $G_{u}$ be an end-grape of $G$ and
let $G_{0}$ be the root component of $G_{u}.$ If $G_{0}$ is non-colorable,
since $G$ is not a grape, it follows that $G_{0}\in\mathfrak{T}$. So, within
$G_{0}$ there must exist an end-grape of $G$ consisting of a single
non-colorable triangular berry with only one exit edge. Then the claim follows
from Lemma \ref{Lemma_nonColorable} and the induction hypothesis. So, let us
assume that $G_{0}$ is colorable. Let $E_{u}=\{uv_{1},\ldots,uv_{p}\}$ and let
$T=G_{u}-E_{u}.$ Notice that $T$ is a tree.

\medskip\noindent\textbf{Case 1:} $T$\emph{ is not colorable}. Since $T$ is a
non-colorable tree, it must be an odd length path. Notice that there exists an
edge $uz$ in $T,$ such that $T^{\prime}=T-uz$ is a collection of one or two
even length paths. Hence, $T^{\prime}$ admits a $2$-liec $\phi_{a,b}%
^{T^{\prime}}.$ Since the degree of $u$ in $T^{\prime}$ is at most one we may
assume $\phi_{a,b}^{T^{\prime}}(u)\subseteq\{a\}.$ Let $E_{u}^{\prime}%
=E_{u}\cup\{uz\}$ and notice that $E_{u}^{\prime}$ contains at least two
edges, so it induces a locally irregular subgraph of $G$ which admits a
$1$-liec $\phi_{c}^{E^{\prime}}.$ Hence, $\phi_{a,b,c}^{u}=\phi_{a,b}%
^{T^{\prime}}+\phi_{c}^{E^{\prime}}$ is a $3$-liec of $G_{u}$ such that
$\phi_{a,b,c}^{u}(u)\subseteq\{a,c\}.$

It remains to consider the root component $G_{0}$ of $G_{u}.$ Since $G_{0}$ is
colorable, the induction hypothesis implies $G_{0}$ admits a $4$-liec
$\phi_{a,b,c,d}^{0}.$ Also, $d_{G_{0}}(u)\leq2$ implies that we may assume
$\phi_{a,b,c,d}^{0}(u)\subseteq\{b,d\}.$ Since $G$ is comprised of $G_{0}$ and
$G_{u}$ which meet at $u,$ from $\phi_{a,b,c,d}^{0}(u)\subseteq\{b,d\}$ and
$\phi_{a,b,c}^{u}(u)\subseteq\{a,c\}$ we conclude that $\phi_{a,b,c,d}%
^{0}+\phi_{a,b,c}^{u}$ is the desired $4$-liec of $G.$

\medskip\noindent\textbf{Case 2:} $T$\emph{ is colorable}. Denote by
$\phi_{a,b,c}^{T}$ a liec of $T,$ and if $T$ is colorable by fewer than $3$
colors, we assume the color $c$ (resp. $b$ and $c$) is not used. We wish to
establish that there exists a liec $\phi_{a,b,c}^{T}$ of $T$ such that
\begin{equation}
\phi_{a,b,c}^{T}(x)\subseteq\{a,b\} \label{For_abT}%
\end{equation}
for each $x\in\{u,v_{1},\ldots,v_{p}\}.$ Now, if $\chi_{%
\rm{irr}%
}^{\prime}(T)\leq2,$ then (\ref{For_abT}) obviously holds. So, let us assume
$\chi_{%
\rm{irr}%
}^{\prime}(T)=3.$ Notice that vertices $u,v_{1},\ldots,v_{p}$ form an
independent set in $T.$ Thus, Observation \ref{Observation_maxDeg} implies we
can choose the rainbow root $z$ in $T$ distinct from $u,v_{1},\ldots,v_{p}$.
Notice that $u$ and all $v_{i}$ for $i\leq p$ belong to at most two distinct
shrubs of $T$ rooted at $z.$ Since $T$ has at least three shrubs rooted at
$z,$ Observation \ref{Observation_maxDeg} implies the color $c$ can be used
only in the shrub not containing vertices $u$ and $v_{i}$ for $i\leq p.$ Thus,
we obtain a $3$-liec of $T$ for which (\ref{For_abT}) holds, and the claim is established.

\medskip\noindent\textbf{Case 2a: }$p\geq2.$ Since $\left\vert E_{u}%
\right\vert =p\geq2$, the set $E_{u}$ induces a locally irregular subgraph of
$G$ which admits a $1$-liec $\phi_{c}^{E}.$ Since (\ref{For_abT}) holds, we
conclude that $\phi_{a,b,c}^{u}=\phi_{a,b,c}^{T}+\phi_{c}^{E}$ is a $3$-liec
of $G_{u}$. It remains to glue $\phi_{a,b,c}^{u}$ with a coloring of $G_{0}.$

Since $G_{0}$ is colorable, the induction hypothesis implies there exists a
$4$-liec $\phi_{a,b,c,d}^{0}$ of $G_{0}.$ If $u$ is $1$-chromatic by
$\phi_{a,b,c,d}^{0},$ we may assume $\phi_{a,b,c,d}^{0}(u)=\{d\},$ so
$\phi_{a,b,c,d}^{0}+\phi_{a,b,c}^{u}$ is the desired $4$-liec of $G.$ If $u$
is $2$-chromatic by $\phi_{a,b,c,d}^{0}$, we may assume that $u_{1}u$ is
$c$-colored, and $u_{2}u$ is $d$-colored. When we glue together $G_{u}$ and
$G_{0}$ with their respective colorings, the only color incident to $u$ in
both graphs is the color $c$. So, it is important to consider $c$-degree of
$u$ and its neighbors in $G_{u}$ and in $G_{0}$, which will be denoted by
$d_{u}^{c}(u)$ and $d_{0}^{c}(u)$ respectively.

Notice that $d_{u}^{c}(u)=p,$ and the $c$-degree of the neighbors of $u$ in
$G_{u}$ is at most one. When $G_{u}$ and $G_{0}$ are glued together, the
$c$-degree of $u$ becomes $p+1$. This will not make $c$-colored edges of
$G_{u}$ incident to $u$ to become locally regular, the only problem is the
edge $u_{1}u$ in $G_{0}.$ If $d_{0}^{c}(u_{1})\not =p+1,$ then $\phi
_{a,b,c,d}^{0}+\phi_{a,b,c}^{u}$ is the desired $4$-liec of $G.$ If $d_{0}%
^{d}(u_{2})\not =p+1,$ we can swap colors $c$ and $d$ in $G_{0}$, so the case
reduces to the previos case $d_{0}^{c}(u_{1})\not =p+1.$ Finally, if
$d_{0}^{c}(u_{1})=d_{0}^{d}(u_{2})=p+1,$ we have to make a small modification
of $\phi_{a,b,c}^{u}.$ Let $\phi_{a,b,c,d}^{\prime u}$ be the coloring of
$G_{u}$ obtained from $\phi_{a,b,c}^{u}$ by changing only the color of
$uv_{1}$ from $c$ to $d.$ Notice that $\phi_{a,b,c,d}^{\prime u}$ is not a
liec of $G_{u},$ since $uv_{1}$ is an isolated edge in color $d$, but
$\phi_{a,b,c,d}^{0}+\phi_{a,b,c,d}^{\prime u}$ is a $4$-liec of $G,$ which
proves the claim.

\medskip\noindent\textbf{Case 2b: }$p=1.$ Let $G_{0}^{\prime\prime}%
=G_{0}+uv_{1}.$ If $G_{0}^{\prime\prime}$ is colorable, by induction
hypothesis it admits a $4$-liec $\phi_{a,b,c,d}^{\prime\prime0}$. Notice that
$d_{G_{0}^{\prime\prime}}(u)=d_{G_{0}}(u)+1\leq3$. Now we argue that $u$ is
$1$- or $2$-chromatic by $\phi_{a,b,c,d}^{\prime\prime0},$ and to see this we
distinguish the case $d_{G_{0}}(u)=1$ and $d_{G_{0}}(u)=2.$ If $d_{G_{0}%
}(u)=1,$ then $d_{G_{0}^{\prime\prime}}(u)=2,$ so the claim obviously holds.
If $d_{G_{0}}(u)=2,$ then $d_{G_{0}^{\prime\prime}}(u)=3,$ but in this case
$u$ is a vertex of a cycle in $G_{0}^{\prime\prime}$ with a pendant edge
$uv_{1}.$ In order for an edge coloring of $G_{0}^{\prime\prime}$ to be
locally irregular, the color of the pendant edge $uv_{1}$ must be the same as
the color of at least one more edge incident to $u$. From this we deduce that
$u$ is $1$- or $2$-chromatic by $\phi_{a,b,c,d}^{\prime\prime0},$ as claimed.
Hence, we may assume $\phi_{a,b,c,d}^{\prime\prime0}(u)\cup\phi_{a,b,c,d}%
^{\prime\prime0}(v_{1})\subseteq\{c,d\}.$ Notice that $G$ consists of
$G_{0}^{\prime\prime}$ and $T$ which meet at $u$ and $v_{1}.$ From
(\ref{For_abT}) and $\phi_{a,b,c,d}^{\prime\prime0}(u)\cup\phi_{a,b,c,d}%
^{\prime\prime0}(v_{1})\subseteq\{c,d\}$ we conclude $\phi_{a,b,c,d}%
^{\prime\prime0}+\phi_{a,b,c}^{T}$ is the desired $4$-liec of $G.$

If $G_{0}^{\prime\prime}$ is not colorable, then the presence of a leaf in
$G_{0}^{\prime\prime}$ implies $G_{0}^{\prime\prime}\in\mathfrak{T}$. Then
again the claim follows from Lemma \ref{Lemma_nonColorable} and the induction hypothesis.
\end{proof}

\section{Concluding remarks}

In this paper we established that all colorable cacti require at most $4$
colors for a locally irregular edge coloring. This is the best possible upper
bound, since there exists the so called bow-tie graph $B,$ which is a
colorable cactus graph with $\chi_{%
\rm{irr}%
}^{\prime}(B)=4.$ This result can be further extended to a claim that every
colorable cactus graph distinct from $B$ requires at most three colors for the
locally irregular edge coloring and a paper with this result is in preparation.

Our argument of this claim is lengthy but uses the same approach as Theorem
\ref{Theorem_main}. The main difference is that in Case 2.a of Theorem
\ref{Theorem_main} we do not have to take much care about $a$- and $b$-degrees
of the neighbors of $u$ in $T$ since we have fourth color $d$ to use it for at
least one of the two edges incident to $u$ in $G_{0}.$ When the fourth color
must not be used, then a great care has to be taken of these $a$- and
$b$-degrees in $T$ because the same colors must be used for both edges
incident to $u$ in $G_{0}$. So, one has to avoid colors $a$ and $b$ in $G_{u}$
to spare them for $G_{0},$ i.e. color all edges of $G_{u}$ incident to $u$ by
$c$ and not just $E_{u}.$ That is not always possible, so special berries and
alternative colorings for them need to be introduced. In light of all this, it
might be helpful for a reader interested into this to consider Theorem
\ref{Theorem_main} as a first step.

\bigskip\noindent\textbf{Acknowledgments.}~~Both authors acknowledge partial
support of the Slovenian research agency ARRS program\ P1-0383 and ARRS
project J1-1692. The first author also the support of Project
KK.01.1.1.02.0027, a project co-financed by the Croatian Government and the
European Union through the European Regional Development Fund - the
Competitiveness and Cohesion Operational Programme.

\end{document}